\documentclass{article}
\usepackage{amssymb}
\usepackage{amsthm}
\usepackage{amsmath}
\usepackage{graphicx}
\usepackage{tikz}

\newtheorem{theorem}{Theorem}[section]
\newtheorem{definition}[theorem]{Definition}
\newtheorem{lemma}[theorem]{Lemma}

\newtheorem{proposition}[theorem]{Proposition}

\newcommand{\C}{\mathcal{C}}
\newcommand{\R}{\mathbb{R}}
\newcommand{\F}{\mathcal{F}}
\newcommand{\D}{\mathcal{D}}
\newcommand{\U}{\mathcal{U}}
\usepackage{fancyhdr}% http://ctan.org/pkg/fancyhdr
\title{A family of convex sets in the plane satisfying the $(4,3)$-property can be pierced by 9 points}
\author{Daniel McGinnis}
\date{\today}

\begin{document}

\maketitle
\thispagestyle{fancy}%
\begin{abstract}
We prove that every finite family of convex sets in the plane satisfying the $(4,3)$-property can be pierced by $9$ points. This improves the bound of 13 proved by 
Gy\'arf\'as, Kleitman, and T\'oth in 2001 \cite{kleitman2001}. 
%In this paper, we show that $9$ points suffice. The main tools used throughout the paper are the KKM Theorem \cite{knaster1929} and the fact that $\tau\leq 2\nu$ for families of separated $2$-intervals \cite{tardos1995}, \cite{kaiser1997}.
\end{abstract}

\section{Introduction}
For positive integers $p\ge q$, a family of sets $\C$ is said to satisfy the $(p,q)$-{\em property} if for every $p$ sets in $\mathcal{C}$, some $q$ sets have a point in common. We say that $\C$ {\em can be pierced
by $m$ points} if there exists a set of size at most $m$ intersecting every element in $\C$. The piercing number $\tau(\C)$ of $\C$ is the minimum $m$ so that $\C$ can be pierced by $m$ points. 

In 1957 Hadwiger and Debrunner \cite{hadwiger1957} conjectured that
for every given positive integers $p\ge q > d$, there exists a constant $c=c_d(p,q)$ such that  every finite family $\C$ of convex sets in $\mathbb{R}^d$ satisfying the $(p,q)$-property  has $\tau(\C) \le c$. This conjecture was proved  by Alon and Kleitman in 1992 \cite{alon1992}. Let $HD_d(p,q)$ denote the least such constant $c_d(p,q)$.

In general, the bounds on $c_d(p,q)$ given by Alon and Kleitman's proof are far from optimal. The first  case where $HD_d(p,q)$ is not known is when $d=2$, $p=4$, and $q=3$. In this case, the constant $c_2(4,3)$ given by the Alon-Kleitman proof is 343, while there is no known example of a family of convex sets in the plane that satisfy the $(4,3)$-property and cannot be pierced by $3$ points.  

In 2001, Gy\'arf\'as, Kleitman, and T\'oth \cite{kleitman2001} proved that $HD_2(4,3)\leq 13$, and since then this bound has seen no improvement. In this paper we prove that $HD_2(4,3)\leq 9$:
\begin{theorem}
If $\C$ is a finite family of convex sets in $\R^2$ then $\tau(\C)\le 9$. 
\end{theorem}
The main tools in the proof are the following two theorems, and a geometrical analysis.

Let $\Delta^{n-1}$ denote the $n$-dimensional simplex on vertex set $e_1,\dots,e_n$.
\begin{theorem}[The KKM theorem \cite{knaster1929}]\label{thm:kkm}
Let $A_1,\dots,A_{n}$ be open sets of $\Delta^{n-1}$ such that for every face $\sigma$ of $\Delta^{n-1}$, $\sigma \subset \bigcup_{e_i \in \sigma} A_i$. Then  $\cap_{i=1}^{n} A_i\neq \emptyset$.
\end{theorem}

A matching in a family of sets $\F$ is a subset of pairwise disjoint sets in $\F$. The matching number $\nu(\F)$ is the maximum size of a matching in  $\F$.

Let $L_1,L_2$ be two homeomorphic copies of the real line. A  {\em $2$-interval} is a union $I_1\cup I_2$, where $I_i$ is an interval on $L_i$.

%A family of sets is {\em intersecting} if every two elements in it intersect.

\begin{theorem}[Tardos \cite{tardos1995}]\label{2intThm}
If  $\F$ is a family of $2$-intervals then  $\tau(\F) \le 2\nu(\F)$.  
\end{theorem}

\section{Using the KKM theorem}

Given a finite family $\C$ of convex sets we may assume that the sets are compact, by considering a set $S$  containing a point in each intersection of sets in $\C$, and replacing every set $C\in \C$ by $C'=\text{conv}\{s\in S \mid s\in C\}.$ 

Let $\C$ be a finite family of compact convex sets satisfying the $(4,3)$ property. We may clearly assume $|\C| \ge 4$. We scale the plane so that  
all the sets in $\C$ are contained in the open unit disk, which we denote by $\U$.
Let $f$ be a parameterization of the unit circle defined by $$f(t)=(cos(2\pi t),-sin(2\pi t))$$ for $t\in [0,1]$. 
For two points $a,b$ in the plane, let $\overline{ab}$ be the line through $a$ and $b$ and let $[a,b]$ be the line segment with $a$ and $b$ as endpoints.

Let $\Delta=\Delta^3$ be the standard $3$-dimensional simplex, and let  
 $x=(x_1,x_2,x_3,x_4)\in \Delta$.  For $1\leq i\leq 4$, define $R^i_x$ to be the interior of the region bounded by the arc along the circle from $f(\sum_{j=1}^{i-1}x_j)$ to $f(\sum_{j=1}^{i}x_j)$ (an empty sum is understood to be $0$) and by the line segments $[(1,0),f(x_1+x_2)]$ and $[f(x_1),f(x_1+x_2+x_3)]$ (see  Figure \ref{fig:regions}). Notice that if $x_i=0$, then $R^i_x=\emptyset$.

\begin{figure}[h!]
    \centering
    \begin{tikzpicture}
    \draw (0,0) circle (3cm);
\filldraw [black] (3,0) circle (2pt) node[right] {$(1,0)$};
\filldraw [black] (95:3cm) circle (2pt) node[above] {$f(x_1+x_2+x_3)$};
\filldraw [black] (145:3cm) circle (2pt) node[left] {$f(x_1+x_2)$};
\filldraw [black] (270:3cm) circle (2pt) node[below] {$f(x_1)$};

\draw (3,0) -- (145:3cm);
\draw (95:3cm) -- (270:3cm);

\filldraw [] (315:1.5cm) node[below] {$R^1_x$};
\filldraw [] (40:2cm) node[below] {$R^4_x$};
\filldraw [] (115:2.5cm) node[below] {$R^3_x$};
\filldraw [] (200:1.5cm) node[below] {$R^2_x$};
\end{tikzpicture}
 \caption{A point $x\in \Delta^3$ corresponds to four regions $R_x^i$.}
    \label{fig:regions}
\end{figure}
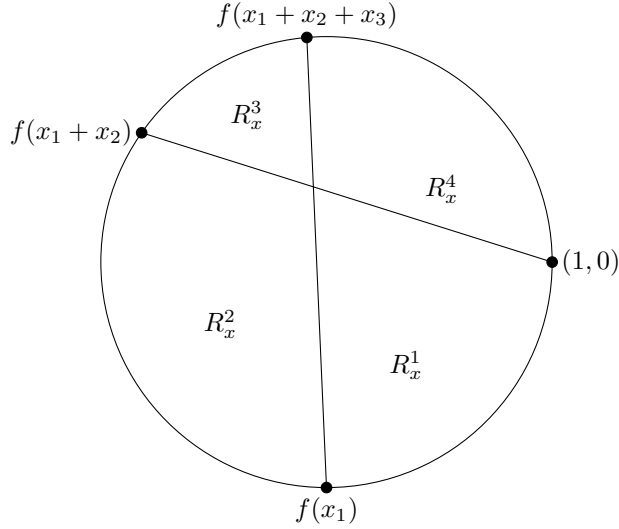

For every $1\leq i\leq 4$ define a  subset  $A_i$  of $\Delta$  as follows: $x\in \Delta^3$ is in $A_i$ if there exists three sets $C_1,C_2,C_3 \in \C$ such that $C_1\cap C_2 \cap C_3\neq \emptyset$ and  $C_j\cap C_k \subset R^i_x$ for all $1\le j< k \le 3$  (see Figure \ref{fig:A_i}). Observe that $A_i$ is open.

%Since $\mathbb{R}^2$ is a normal topological space, there exists an $\epsilon$-neighborhood around the closed set $[(1,0),f(x_1+x_2)]\cup [f(x_1),f(x_1+x_2+x_3)]$ that is disjoint from the union of the pairwise intersections of $C_1,C_2,C_3$. So if $x\in A_i$ then the set of points $y$ whose entries are within $\epsilon/4$ of the corresponding entries of $x$ is also in $A_i$.

%If there exists some $x\in \Delta^3$ such that $x\notin \cup_{i=1}^4A_i$, then we are done as shown below.

%First, we show in Proposition \ref{easy} that if there exists $x\in \Delta$ that does not belong to $\bigcup A_i$, then $\C$ can be pierced by $8$ points.  

For every $x\in \Delta$ and $C\in \C$ let $I_C$ be the (possibly empty) 2-interval 
$$(C\cap [(1,0),f(x_1+x_2)])\cup (C\cap [f(x_1),f(x_1+x_2+x_3)]).$$ 

\begin{lemma}\label{easylem}
Suppose there exists  $x\in \Delta \setminus \Big( \cup_{i=1}^4A_i\Big)$. Then there exist two points $a,b$ such that if $a,b \notin C$ then $I_C\neq \emptyset$. 
\end{lemma}
\begin{proof}
Note that since $\C$ does not contain three pairwise non-intersecting sets, at most two of the regions $R^i_x$ can contain a set in $\C$. 

We claim for every $i\le 4$, the region $R^i_x$ contains at most two sets in $\C$. Indeed, 
assume to the contrary that $R^i_x$ contains three  sets  $C_1,C_2,C_3 \in \C$. Then $C_1\cap C_2 \cap C_3=\emptyset$ since $x\notin A_i$. Applying the $(4,3)$ property to $C_1,C_2,C_3$ and some additional set $F\in \C$, we obtain that $C_j\cap C_k \cap F \neq \emptyset$ for some $1\le j<k \le 3$, and all pairwise intersections of $C_j,C_k,F$ are contained in $R^i_x$, contradicting   $x\notin A_i$. 

Since there are at most two regions containing sets in $\C$, the lemma is proved. 
\end{proof}

\begin{figure}[h!]
    \centering
    \begin{tikzpicture}
\draw (0,0) circle (3cm);
\filldraw [black] (3,0) circle (2pt) node[right] {$(1,0)$};
\filldraw [black] (95:3cm) circle (2pt) node[above] {$f(x_1+x_2+x_3)$};
\filldraw [black] (145:3cm) circle (2pt) node[left] {$f(x_1+x_2)$};
\filldraw [black] (270:3cm) circle (2pt) node[below] {$f(x_1)$};

\draw (3,0) -- (145:3cm);
\draw (95:3cm) -- (270:3cm);

\draw (315:1cm) circle [x radius=2.5cm, y radius=3mm, rotate=30];
\draw (325:.5cm) circle [x radius=5mm, y radius=2.4cm, rotate=30];
\draw (335:1.25cm) circle [x radius=5mm, y radius=2cm, rotate=-15];
\end{tikzpicture}
 \caption{Three sets $C_1,C_2,C_3\in C$ with $C_1\cap C_2 \cap C_3\neq \emptyset$ and $C_j\cap C_k \subset R^1_x$ for all $1\le j< k \le 3$, implying $x=(x_1,x_2,x_3,x_4) \in A_1$.}
    \label{fig:A_i}
\end{figure}
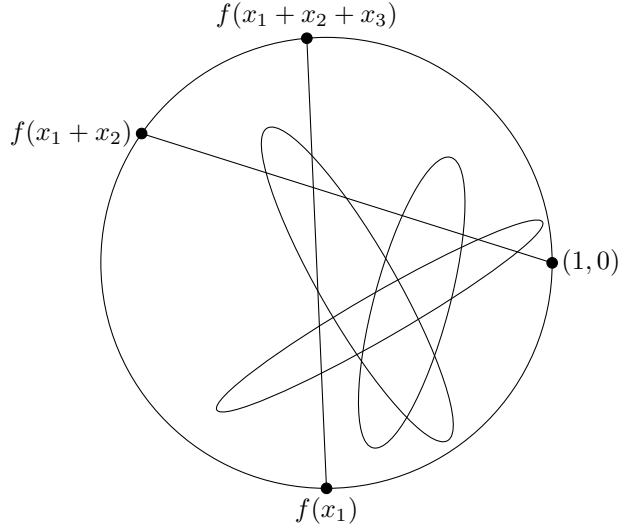
%If only $R^1_x$ for instance contains a set in $\C$, then by the above, the sets in $R^1_x$ can be pierced by two points since at most two sets in $\C$ are contained in $R^1_x$. 

%If say $R^1_x$ and $R^2_x$ contain sets in $\C$, then $R^1_x$ or $R^2_x$ contain only one set in $\C$ by the $(4,3)$ property. Assume $R^1_x$ contains only one set in $\C$, then this set can be pierced by one point. If $R^2_x$ contains two sets in $\C$, then these sets must intersect because $\C$ does not have three pairwise non-intersecting sets. Thus, the sets in $R^2_x$ can be pierced by $1$ point. 

%This completes the proof.

\begin{theorem}\label{easy}
If there exists  $x\in \Delta \setminus \Big( \cup_{i=1}^4A_i\Big)$,  then $\tau(\mathcal{C})\le 8$.
\end{theorem}
\begin{proof}
Let $\D = \{C\in \C \mid I_C \neq \emptyset\}$. 
We will show that $\tau(\D)\le 6$. Together with Lemma \ref{easylem} this will imply the theorem. 

Let $\mathcal{I}=\{I_C:C\in \mathcal{D}\}$. 
%We will show that $\nu(\mathcal{I})\le 3$, and hence by Theorem \ref{2intThm}, $\tau(\mathcal{I})\le 6$. Since for every $C\in \D$, $I_C\neq \emptyset$, this will imply $\tau(\mathcal{D})\le 6$, and hence $\tau(\mathcal{C})\le 8$.
Let $C_1,C_2,C_3,C_4\in \mathcal{D}$ be four sets. Some three, say $C_1,C_2,C_3$, intersect by the $(4,3)$-property. Since $x\notin \cup_{i=1}^4 A_i$, the intersection of  two of these three sets, say $C_1 \cap C_2$, must intersect  either $[(1,0),f(x_1+x_2)]$ or $[f(x_1),f(x_1+x_2+x_3)]$. In other words, $I_{C_1} \cap I_{C_2} \neq \emptyset$.  This shows that $\mathcal{I}$ has no four pairwise disjoint elements, implying $\nu(\mathcal{I})\le 3$. Thus by Theorem  \ref{2intThm}, $\tau(\D)\le \tau(\mathcal{I})\le 6$.
\end{proof}

 By Theorem \ref{easy} we may assume that  
$\Delta \subset \cup_{i=1}^4 A_i$. We claim  that  in this case the sets $A_1,\dots,A_4$ satisfy the conditions of Theorem \ref{thm:kkm}. Indeed, suppose  $\sigma= \ \text{conv}\{e_i:i\in I\}$ is a face of $\Delta$, for some $I\subset [4]$, and let $y\in \sigma$. Then for all $j\in [4]\setminus I$, we have  $R_y^j=\emptyset$, implying  $y\notin A_j$. Since $y \in \cup_{i=1}^4 A_i$, we have that $y \in \bigcup_{i\in I} A_i$.
Thus by Theorem \ref{thm:kkm} we have:

\begin{theorem}\label{hard}
If $\Delta\subset \cup_{i=1}^4 A_i$, then there exists $x\in \bigcap_{i=1}^4 A_i$.
\end{theorem}

%By Theorems \ref{easy} and \ref{hard}, it suffices to consider the case that there exists some $x\in \cap_{i=1}^4A_i$. 

For the rest of the paper we fix $x \in \bigcap_{i=1}^4 A_i$. 
Let $R^i_x=R^i$, and let $f_0=(1,0)$, $f_1=f(x_1)$, $f_2=f(x_1+x_2)$, and $f_3=f(x_1+x_2+x_3)$. Let $c$ be the intersection point of $[(1,0),f_2]$ and $[f_1,f_3]$, and let $\mathcal{C}' = \{C\in \C \mid c \notin C \}$. We use $\overline{R^i}$ to denote the topological closure of $R^i$.

\begin{proposition}\label{prop1}
If $C\in \mathcal{C}'$, then there exists some $i$ for which $C\cap R^i=\emptyset.$ 
\end{proposition}
\begin{proof}
Assume  $C$ has a point $p_i$ in each $R^i$. Then since $C$ is convex,  it contains the points  $q_1 =  [p_1,p_2] \cap [f_1,f_3]$ and  $q_2=[p_3,p_4]\cap [f_1,f_3]$. Since $q_1$ and $q_2$ lie in two different hyperplanes defined by the line $\overline{f_0f_2}$, $C$ must contain $c$, a contradiction.
\end{proof}

%Assume that $C$ has a point $p_i$ in each $R_x^i$. Then since $C$ is convex,  it contains the points  $q_1 =  [p_1,p_2] \cap [f_1,f_2]$, and  $q_2=[p_3,p_4]\cap [f_1,f_2]$. Since $q_1$ and $q_2$ lie below and above $c$ on $[f_1,f_2]$ respectively, we have that $c\in [q_1,q_2]$, so $F$ must contain $c$, a contradiction.

Let $\mathcal{C}_i$ denote the family of sets in $\mathcal{C}'$ that are disjoint from $R^i$. By Proposition \ref{prop1}, we have $\mathcal{C}'=\cup_{i=1}^4\mathcal{C}_i$. In the  remainder of the paper we prove the following:

\begin{theorem}\label{main2}
For every $i\le 4$, $\tau(\mathcal{C}_i) \le 2$.
\end{theorem}

 This will imply that  $\mathcal{C}$ can be pierced by $9$ points: two points for each $\mathcal{C}_i$ and the point  $c$.

\section{Piercing $\mathcal{C}_i$ by two points}
In this Section we prove Theorem \ref{main2}. Without loss of generality 
we prove the theorem for  $\mathcal{C}_1$.

\subsection{Preliminary definitions and observations}
Let $C_1,C_2,C_3 \in \C$ be the three sets witnessing the fact that $x\in A_1$; so $C_1\cap C_2 \cap C_3\neq \emptyset$ and $C_j\cap C_k \subset R^1$ for all $1\le j< k \le 3$. 

If there are two sets $F_1,F_2\in \mathcal{C}_1$ that do no intersect, then $F_1,F_2,C_1,C_2$ do not satisfy the $(4,3)$-property. Thus every two sets in $\mathcal{C}_1$ intersect. Also, if for some $1\le i \le 3$ we have $C_i\subset R^1$, then again by the  $(4,3)$-property  every three sets in $\mathcal{C}_1$ have a common point. This implies by Helly's theorem \cite{helly} that $\tau(\mathcal{C}_1)=1$.  So we may assume that no $C_i$ is contained in $R^1$. 

Let $L_1$ be the line  $\overline{f_1f_3}$ and let $L_2$ be the line  $\overline{f_0f_2}$. Let $L_1^+$ and $L_1^-$ be the closed halfspaces containing $f_0$ and $f_2$ respectively, and let $L_2^+$ and $L_2^-$ be the closed halfspaces containing $f_3$ and $f_1$ respectively (see Figure \ref{fig:Llines}). 

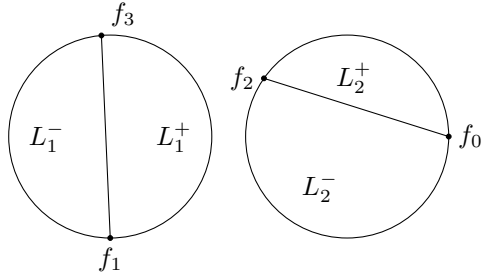
\begin{figure}[h!]
    \centering
    \vskip10pt
    \begin{tikzpicture}[scale=.45]
    \draw (0,0) circle (3cm);
    \filldraw [black] (95:3cm) circle (2pt) node[above] {$\ \ \ \ \ f_3$};
    \filldraw [black] (270:3cm) circle (2pt) node[below] {$f_1$};
    \filldraw [black] (160:2cm) node[below] {$L_1^-$};
    \filldraw [black] (20:2cm) node[below] {$L_1^+$};
    \draw (95:3cm) -- (270:3cm);
    \end{tikzpicture}
    \begin{tikzpicture}[scale=.45]
    \draw (0,0) circle (3cm);
    \filldraw [black] (270:3cm)  node[below] {\phantom{$f_1$}};
    \filldraw [black] (3,0) circle (2pt) node[right] {$f_0$};
    \filldraw [black] (145:3cm) circle (2pt) node[left] {$f_2$};
    \filldraw [black] (270:1.5cm) node[left] {$L_2^-$};
    \filldraw [black] (60:2cm) node[left] {$L_2^+$};
    \draw (3,0) -- (145:3cm);
    \end{tikzpicture}
    \caption{The regions $L_1^+, L_1^-$ and $L_2^+,L_2^-$.}
    \label{fig:Llines}
\end{figure}

By our assumption $C_i$ is not contained in $R^1$ for $1\le i\le 3$, and thus $C_i \cap (\mathbb{R}^2\setminus R^1)$ has at least one non-empty connected component.  
The next proposition shows that  $C_i \cap (\mathbb{R}^2\setminus R^1)$ has at most two  connected components. 

\begin{proposition}
For every $1\le i \le 3$, the set $C_i \cap (\mathbb{R}^2\setminus R^1)$ has at most two connected components. Moreover, if $C_i \cap (\mathbb{R}^2\setminus R^1)$ has two components, then the components are $C_i\cap \overline{R^2}$ and $C_i\cap \overline{R^4}$ and hence are convex. 
\end{proposition}
\begin{proof}
If $C_i$ contains $c$, then $C_i \cap (\mathbb{R}^2\setminus R^1)$ has one component because the line segment from any point in $\mathbb{R}^2\setminus R^1$ to $c$ is contained in $\mathbb{R}^2\setminus R^1$. So assume $C_i$ does not contain $c$. 
Then it must have a point in either $\overline{R^4}$ or $\overline{R^2}$, without loss of generality, in $\overline{R^4}$.

Suppose $C_i$ contains a point in $\overline{R^3}$.
 Since $C_i$ does not contain $c$ but contains points in the three regions $\overline{R^1},\overline{R^4},\overline{R^3}$, then by Proposition \ref{prop1} 
it cannot contain a point in $\overline{R^2}$. 
Thus $C_i\cap (\mathbb{R}^2\setminus R^1)=C_i\cap (\overline{R^4}\cup \overline{R^3})$. This means that $C_i\cap (\mathbb{R}^2\setminus R^1)$ is an intersection of two convex sets, hence it is convex and has only one component.

Thus, if $C_i\cap (\mathbb{R}^2\setminus R^1)$ has more than one component, then $C_i$ does not have a point in $\overline{R^3}$. In this case the components of $C_i\cap (\mathbb{R}^2\setminus R^1)$ are $C_i\cap \overline{R^2}$ and $C_i\cap \overline{R^4}$ both of which are convex.
\end{proof}

Let $Z=[f_1,c]\cup [c,f_0]$. We think of $Z$ as a segment starting at $f_1$ and ending at $f_0$. Thus a point $a\in Z$ comes before a point $b\in Z$ if the distance from $a$ to $f_1$ on $Z$ is not larger than the distance from $b$ to $f_1$ on $Z$.

Let  $I_i^1=C_i \cap [f_1,c]$,  $I_i^2=C_i\cap [c,f_0]$, and $I_i=C_i\cap Z$.

For any interval (i.e., connected set) $I$ on $Z$, let $l(I)$ be the endpoint of $I$ that 
comes first on $Z$, and let $r(I)$ be the other endpoint.
Given a convex set $C$ and a point $p$ on the boundary of $C$, a \textit{supporting line for $C$} at $p$ is a line $L$ passing through $p$ that contains $C$ in one of the closed  halfspaces defined by $L$. 
For $1\le i\le 3$, let $C_i'=C_i\cap (\mathbb{R}^2\setminus R^1)$.

\begin{definition}
Let $A$ be a connected component of $C_i'$, and let $I=A\cap Z$ (so $I$ is an interval on $Z$). We define  supporting lines $S_i^l(I)$ and $S_i^r(I)$ for $C_i$ at the points $l(I)$ and $r(I)$, respectively, as follows:
\begin{itemize}
    \item If $A\subset L_j$ for some $j\in \{1,2\}$, 
    then $S_i^l(I)=S_i^r(I)=L_j$. 
    \item If $r(I)=c$, $A$ lies in $L_2^-$, and $A\not\subset L_1$, then $S_i^r(I)=L_2$.  
    \item If $l(I)=c$, $A$ lies in $L_1^+$, and $A\not\subset L_2$, then $S_i^l(I)=L_1$.
    \item Otherwise, set $S_i^l(I)$ and $S_i^r(I)$ to be any supporting line for $C_i$ at the point $l(I)$ and $r(I)$, respectively.
\end{itemize}
\end{definition}

\begin{definition}\label{def2comp}
Assume $C_i'$ has two components $A_1=C_i'\cap \overline{R^2}$ and $A_2=C_i'\cap \overline{R^4}$. We define $S_i'$ to be a  piece-wise linear curve as follows:
\begin{itemize}
    \item If $A_1\subset L_1$ and $A_2\subset L_2$, then $$S_i'=[f_1,r(I_i^1)]\cup [r(I_i^1),l(I_i^2)]\cup [l(I_i^2),f_0].$$
%    \vspace{-1cm}
    \item If $A_1\subset L_1$ and $A_2\not\subset L_2$, then $$S_i'=[f_1,r(I_i^1)]\cup [r(I_i^1),l(I_i^2)]\cup (S_i^l(I_i^2)\cap \overline{R^4}).$$
 %   \vspace{-1cm}
    \item If $A_1\not\subset L_1$ and $A_2\subset L_2$, then $$S_i'=(S_i^r(I_i^1)\cap \overline{R^2})\cup [r(I_i^1),l(I_i^2)]\cup [l(I_i^2),f_0].$$
  %  \vspace{-1cm}
    \item If $A_1\not\subset L_1$ and $A_2\not\subset L_2$, then $$S_i'=(S_i^r(I_i^1)\cap \overline{R^2})\cup [r(I_i^1),l(I_i^2)]\cup (S_i^l(I_i^2)\cap \overline{R^4}).$$
\end{itemize}
\end{definition}

Note that in all the cases $S_i'$ lies in the closed halfspace defined by the line $\overline{r(I_i^1)l(I_i^2)}$ containing $f_0$ and $f_1$.

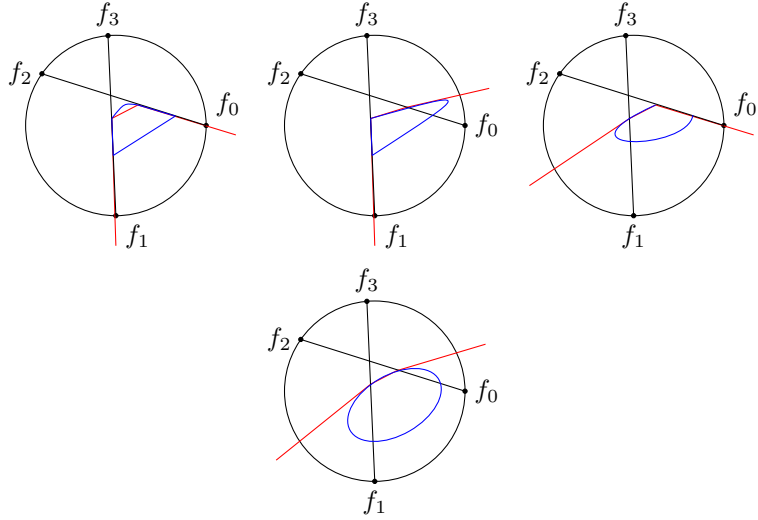
\begin{figure}[h!]
    \centering
    
\begin{tikzpicture}[scale=.4]
\draw (0,0) circle (3cm);
\filldraw [black] (3,0) circle (2pt) node[above right] {$f_0$};
\filldraw [black] (95:3cm) circle (2pt) node[above] {$ f_3$};
\filldraw [black] (145:3cm) circle (2pt) node[left] {$f_2$};
\filldraw [black] (270:3cm) circle (2pt) node[below right] {$f_1$};

\draw [red] (300:-.27cm) -- (43:1cm);
\draw [red] (43:1cm) -- (355.5:4cm);
\draw [red] (300:-.27cm) -- (270:4cm);

\draw (3,0) -- (145:3cm);
\draw (95:3cm) -- (270:3cm);

\draw [blue] (43:1cm) .. controls (70:.8cm) .. (300:-.27cm);
\draw [blue] (43:1cm) -- (9:2cm);
\draw [blue] (300:-.27cm) -- (265:1cm);
\draw [blue] (9:2cm) -- (265:1cm);

\end{tikzpicture}
\begin{tikzpicture}[scale=.4]
\draw (0,0) circle (3cm);
\filldraw [black] (3,0) circle (2pt) node[right] {$f_0$};
\filldraw [black] (95:3cm) circle (2pt) node[above] {$ f_3$};
\filldraw [black] (145:3cm) circle (2pt) node[left] {$f_2$};
\filldraw [black] (270:3cm) circle (2pt) node[below] {$\ \ \ \ \ f_1$};

\draw [red] (300:-.27cm) -- (30:1.2cm);
\draw [red] (30:1.2cm) -- (18:4cm);
\draw [red] (300:-.27cm) -- (270:4cm);

\draw (3,0) -- (145:3cm);
\draw (95:3cm) -- (270:3cm);

\draw [blue] (300:-.27cm) -- (265:1cm);
\draw [blue] (300:-.27cm) .. controls (20: 3.5cm) .. (265:1cm);

\end{tikzpicture}
\begin{tikzpicture}[scale=.4]
\draw (0,0) circle (3cm);
\filldraw [black] (3,0) circle (2pt) node[above right] {$f_0$};
\filldraw [black] (95:3cm) circle (2pt) node[above] {$ f_3$};
\filldraw [black] (145:3cm) circle (2pt) node[left] {$f_2$};
\filldraw [black] (270:3cm) circle (2pt) node[below] {$f_1$};

\draw [red] (300:-.27cm) -- (43:1cm);
\draw [red] (43:1cm) -- (355.5:4cm);
\draw [red] (300:-.27cm) -- (210:4cm);

\draw (3,0) -- (145:3cm);
\draw (95:3cm) -- (270:3cm);

\draw [blue] (43:1cm) .. controls (198:3cm) and (335:2cm) .. (10:2cm);

\end{tikzpicture}
\begin{tikzpicture}[scale=.4]
\draw (0,0) circle (3cm);
\filldraw [black] (3,0) circle (2pt) node[right] {$f_0$};
\filldraw [black] (95:3cm) circle (2pt) node[above] {$ f_3$};
\filldraw [black] (145:3cm) circle (2pt) node[left] {$f_2$};
\filldraw [black] (270:3cm) circle (2pt) node[below] {$f_1$};

\draw [red] (300:-.27cm) -- (43:1cm);
\draw [red] (43:1cm) -- (23:4cm);
\draw [red] (300:-.27cm) -- (215:4cm);

\draw (3,0) -- (145:3cm);
\draw (95:3cm) -- (270:3cm);

\draw [blue] (325:.8cm) circle [x radius=10mm, y radius=1.7cm, rotate=120];

\end{tikzpicture}
    \caption{Definition \ref{def2comp}. The blue set is $C_i$, and the red curve is $S_i'$. }
    %Here, $S_i'$ is given when $A_1\subset L_1$, $A_2\subset L_2$, $A_1\not\subset L_1$, $A_2\subset L_2$, and $A_1\not\subset L_1$, $A_2\not\subset L_2$, respectively.}
    \label{fig:my_label}
\end{figure}

\subsection{Five lemmas}

If $A_i$ is a component of $C_i'$ and  $I_i=A_i\cap Z$, we say that {\em $I_i$ comes before $I_j$ on $Z$} if the point $r(I_i)$ comes before $l(I_j)$ on $Z$.

\begin{lemma}\label{sep}
Let $1\le i\neq j\le 3$. Let $A$ be a component of $C_i'$ and $I=A\cap Z$. Let $B$ be a component of $C_j'$ and  $J=B\cap Z$. Suppose that $F \subset \U$ is a convex set such that $F\cap R^1=\emptyset$, $F\cap A\neq \emptyset$ and $F\cap B \neq \emptyset$. 
If $I$ comes before $J$ on $Z$, then $F\cap S_i^r(I)\neq \emptyset$  and  $F\cap S_j^l(J) \neq \emptyset$.
\end{lemma}
\begin{proof}
Observe that if $A$ is a line segment or a single point, 
 then $A\subset S_i^r(I)$ and we are done. This is true since if $A\subset L_t$ for some $t\in \{1,2\}$  then $S_i^r(I)$ is $L_t$, and if $A$ is a line segment  not contained in $L_1$ or $L_2$, then $C_i$ must be a line segment so $A\subset S_i^r(I)$. So we may assume that $A$ is not a line segment, and in particular,   $I$ consists of more than one point.

We will show that $F\cap S_i^r(I)\neq \emptyset$. The other statement follows similarly. 

{\bf Case 1.} $r(I)\in [c,f_0]$.
We claim that $B$ does not have a point $L_2^-$. Indeed, by definition $B$ does not have a point in $R^1$. If $B$ has a point $p$ in $\overline{R^2}$, then for a point $q\in C_j\cap R^1$, we have that the line segment $[p,q]$ crosses $Z$ in $y\in [f_1,c]$. Therefore, $y$ comes before $l(J)$ on $Z$, a contradiction.

Now, if $S_i^r(I)=L_2$, then since  $A\subset L_2^-$ and $B \subset L_2^+$ then $F$ must intersect $L_2$  as needed. 
So we can assume $S_i^r(I)\neq L_2$. Let $H$ be the closed halfspace defined by $S_i^r(I)$ containing $C_i$. If $r(I)\in (c,f_0]$, then since $I$ consists of more than one point, $H$ contains a point of $[c,f_0]$ that comes before $r(I)$ on $Z$, so every point on $Z$ coming after $r(I)$ must lie in the complement of $H$. If $r(I)=c$, then since $S_i(r(I))\neq L_2$, we have that $A$ has a point $y$ in $L_2^+\setminus L_2$ (this follows from the definition of $S_i^r(I)$ when $r(I)=c$). 
 
 Note that $y\in \overline{R^3}$. This is true because  if $A$ has a point in $\overline{R^4}$ other than $c$, then $A$ must contain a point in $(c,f_0]$. By the convexity of $C_i$ and the fact that $I$ contains more than one point, we can conclude that $A$ has a point in $L_2\cap L_1^-$ other than $c$. Again, this implies that every point on $Z$ coming after $c$ lies in the complement of $H$. Now, if $B$ has a point $p$ in $H$, then for a point $q\in C_i\cap C_j\subset H\cap R^1$, we have that the segment $[p,q]$ crosses $Z$ in $H$. This is a contradiction since every point of $J$ lies in the complement of $H$.

{\bf Case 2.} $r(I)\in [f_1,c)$. If $S_i^r(I)=L_1$, then $A\subset S_i^r(I)$ and we are done. 

Otherwise, take $H$ to be the closed halfspace defined by $S_i^r(I)$ containing $C_i$. If $S_i^r(I)$ does not intersect $(c,f_0]$, then every point after $r(I)$ on $Z$ lies in the complement of $H$ and we can apply the previous argument to conclude that $B$ cannot contain a point in $H$. If $S_i^r(I)$ intersects $(c,f_0]$, then the set $(H\setminus R^1)\cap \U$ has two connected components. Let $H'$ be the component containing $A$. Every point that comes after $r(I)$ on $Z$ lies in the complement of $H'$, so a similar argument to the one used above shows that $B$ cannot contain a point in $H'$. Thus $F$ must intersect $S_i^r(I)$. 
\end{proof}

Similar arguments can be applied to prove the following lemma.

\begin{lemma}\label{ssep}
Let $1\le i\neq j\le 3$. Assume that $C_i'$ has two components $A_1=C_i'\cap \overline{R^2}$ and $A_2=C_i'\cap \overline{R^4}$. Let $B$ be a component of $C_j'$ and  $J=B\cap Z$. Suppose $F \subset \U$ is a convex set 
such that $F\cap R_x^1=\emptyset$ and $F\cap B\neq \emptyset$. If $F\cap A_1\neq \emptyset$ and $J$ comes after $I_i^1$ on $Z$, or if $F\cap A_2\neq \emptyset$ and $J$ comes before $I_i^2$ on $Z$, then $F\cap S_i'\neq \emptyset$.
%Then the following hold:
%\begin{itemize}
    %\item If $F\cap A_1\neq \emptyset$ and $J$ comes after $I_i^1$ on $Z$ then $F\cap S_1'\neq \emptyset$.
    %\item If $F\cap A_2\neq \emptyset$ and $J$ comes after $I_i^2$ on $Z$ then $F\cap S_2'\neq \emptyset$.
%\end{itemize}
\end{lemma}

We say that a set $F$ {\em lies below a line $L$ in $\overline{R^2}$} 
if $F$ does not intersect $L$, $F\subset \overline{R^2}$, and $F$ lies on the side of $L$ containing $f_1$. Note that if $F$ lies below a $L_1$ in $\overline{R^2}$ then $F$ must be empty. Similarly, we say that $F$ {\em lies below $L$ in $\overline{R^4}$} if $F$ does not intersect $L$, $F\subset \overline{R^4}$, and lies on the side of $L$ that contains $f_0$. If $L=L_2$, then again $F=\emptyset$.

\begin{lemma}\label{sssep}
Assume that for $1\le i\neq j\le 3$,  $C_i'$ and $C_j'$ both have two components:  $A_1=C_i'\cap \overline{R^2}$, $A_2=C_i'\cap \overline{R^4}$, $B_1=C_j'\cap \overline{R^2}$, and $B_2=C_j'\cap \overline{R^4}$. 
Assume that $I_i^1$ comes before $I_j^1$ on $Z$ and $I_i^2$ comes before $I_j^2$ on $Z$. Let $F_1,F_2\in \mathcal{C}_1$, and write $F=F_1\cap F_2$. Then one of the following holds: 
\begin{itemize}
\item If $F$ intersects $C_i'$ and $C_j'$ or $F$ intersects $B_2$ then $F \cap S_i^r(I_i^1) \neq \emptyset$.
\item  If $F$ intersects $C_i'$ and $C_j'$ or $F$ intersects $A_1$, then $F \cap S_j^l(I_j^2)\neq \emptyset$.
\end{itemize}
\end{lemma}
\begin{proof}
First note that $A_1$ lies below $S_j^l(I_j^2)$ in $\overline{R^2}$ and $B_2$ lies below $S_i^r(I_i^1)$ in $\overline{R^4}$ %$S_i^r(I_i^1)\cap L_2^+$ lies in $\overline{R^4}$ and $S_j^l(I_j^2)\cap L_1^-$ lies in $\overline{R^2}$.
For instance, if $A_1$ contains a point $p$ lying on or above $S_j^l(I_j^2)$, then the segment $[p,l(I_i^2)]$ crosses $[f_1,c]$ on or above $S_j^l(I_j^2)$, contradicting the fact that $I_i^1$ comes before $I_j^1$ on $Z$. A similar argument applies to the corresponding statement for $B_2$. 

Assume for contradiction that there exists two sets $F_1,F_2\in \mathcal{C}_1$ such that $F_1\cap F_2$ lies below $S_j^l(I_j^2)$ in $\overline{R^2}$ and  $A_1\cap (F_1\cap F_2)\neq \emptyset$, and two sets $F_3,F_4\in \mathcal{C}_1$ such that $F_3\cap F_4$ lies below $S_i^r(I_i^1)$ in $\overline{R^4}$ and $B_2\cap (F_1\cap F_2)\neq \emptyset$. By the $(4,3)$-property, three sets out of $F_1,F_2,F_3,F_4$ have a common point. If $F_1,F_2,F_3$ intersect, then $F_3$ intersects $B_2$ and has a point below $S_j^l(I_j^2)$ in $\overline{R^2}$, which implies that $F_3$ has a point in $R^1$, a contradiction. Similarly, $F_1,F_3,F_4$ cannot intersect. Therefore,  there is no pair of sets in $\mathcal{C}_1$ whose intersection intersects $B_2$ and lies below $S_i^r(I_i^1)$ in $\overline{R^4}$, or there is no pair of sets in $\mathcal{C}_1$ whose intersection intersects $A_1$ and lies below $S_j^l(I_j^2)$ in $\overline{R^2}$.

Assume that there is no pair of sets in $\mathcal{C}_1$ whose intersection lies below $S_i^r(I_i^1)$ in $\overline{R^4}$ and intersects $B_2$, and take $L=S_i^r(I_i^1)$. Let $F$ be the intersection of any pair of sets in $\mathcal{C}_1$. If $F$ intersects $B_2$, then by the above $F$ does not lie below $L$ in $\overline{R^4}$. This implies that $F$ intersects $L$ since $B_2$ lies below $L$ in $\overline{R^4}$. If $F$ intersects $B_1$ and $A_1$, then $F$ intersects $L$ by Lemma \ref{sep}. If $F$ intersects $B_1$ and $A_2$, then $F$ intersects $L$ since  $B_1$ lies in the halfspace defined by $L$ that does not contain $C_i$.

If there is no pair of sets in $\mathcal{C}_1$ whose intersection lies below $S_j^l(I_j^2)$ in $\overline{R^2}$ and intersects $A_1$, then a similar argument shows that the corresponding statements follow for $S_j^l(I_j^2)$.
\end{proof}

%Our strategy will be to find two lines $T_1$ and $T_2$ of the form $S_i^r(I)$, $S_i^l(I)$, or $S_i'$, and show that the collection of $2$-intervals $\mathcal{I}=\{(F\cap T_1)\cup (F\cap T_2):F\in \mathcal{C}_1\}$ has matching number $1$ and hence can be pierced by $2$ points. To make this argument, we need to verify that $(F\cap T_1)\cup (F\cap T_2)$ is a $2$-interval for all $F\in \mathcal{C}_1$. Of course, the intersection of a convex set with a line is always an interval. We show below that the intersection of any $F\in \mathcal{C}_1$ with $S_i'$ is always an interval, and this implies that $(F\cap T_1)\cup (F\cap T_2)$ is always a $2$-interval.

\begin{lemma}\label{sline}
If $F\in \mathcal{C}_1$ and $C_i'$ has two components, then $F\cap S_i'$ is an interval.
\end{lemma}
\begin{proof}
Clearly, $F\cap (S_i'\cap \overline{R^2})$ and $F\cap (S_i'\cap \overline{R^4})$ are intervals, and $F\cap S_i'=(S_i'\cap \overline{R^2}) \cup (S_i'\cap \overline{R^4})$,  so it suffices to show that $F$ cannot intersect both $S_i'\cap \overline{R^4}$ and $S_i'\cap \overline{R^2}$. 
Suppose it does.
Let $T$ be the line passing through $r(I_i^1)$ and $l(I_i^2)$. By the definition of $S_i'$, both $S_i'\cap \overline{R^2}$ and $S_i'\cap \overline{R^4}$ lie on the closed halfspace defined by $T$ containing $f_0$ and $f_1$. Since $F$ is convex, this implies that $F$ has a point in $R^1$, a contradiction. 
%If $(C_i'\cap \overline{R^4})\subset L_2$, then $S_i'\cap \overline{R^4}$ lies below $T$ in $\overline{R^4}$. Otherwise, since $[r(I_i^1),l(I_i^2)]$ is contained in $C_i$, $S_i(l(I_i^2))$ contains $[r(I_i^1),l(I_i^2)]$ in one of its resulting closed halfspaces, so $S_i'\cap \overline{R^4}$ lies on or below $T$ in $\overline{R^4}$. Similarly, $S_i'\cap \overline{R^2}$ lies below $T$ in $\overline{R^2}$. Therefore, any convex set intersecting both $S_i'\cap \overline{R^4}$ and $S_i'\cap \overline{R^2}$ must contain a point in $R^1$.
\end{proof}

\begin{lemma}\label{pairs}
Let $F_1,F_2\in \mathcal{C}_1$, then $F_1\cap F_2$ intersects at least two of $C_1,C_2,C_3$.
\end{lemma}
\begin{proof}
Suppose $F_1\cap F_2$ does not intersect $C_1$. Since $(C_1\cap C_2)\subset R^1$ and $F\cap R^1=\emptyset$,  by the $(4,3)$-property for the sets $C_1,C_2,F_1,F_2$, we have that $C_2$ must intersect $F_1\cap F_2$. Similarly, $C_3$ intersects $F_1\cap F_2$. 
\end{proof}

\subsection{Proof of Theorem \ref{main2}}
We wish to show that $\tau(\C_1)\le 2$.
We split into four cases. In each case and subcase, we find two homeomorphic copies of the real line $T_1$ and $T_2$, and show that the family of $2$-intervals $\mathcal{I}=\{(F\cap T_1)\cup (F\cap T_2):F\in \mathcal{C}_1\}$ satisfies $\nu(\mathcal{I})=1$. By Theorem \ref{2intThm}, this implies $\tau(\mathcal{I})\le 2$. 
The curves $T_1, T_2$ will be 
of the form $S_i^r(I)$, $S_i^l(I)$, or $S_i'$, and  Lemma \ref{sline} ensures that $\mathcal{I}$ is indeed a family of $2$-intervals. Recall that $I_i^1=C_i \cap [f_1,c]$,  $I_i^2=C_i\cap [c,f_0]$, and $I_i=C_i\cap Z$. %Lemmas \ref{sep}, \ref{ssep}, \ref{sssep}, and \ref{pairs} will be used to show that $\mathcal{I}$ indeed has matching number $1$. \\

\bigskip\noindent
\textbf{Case 1.} $C_i'$ has one component for each $i$ (see Figure \ref{fig:case1}).

Notice in this case each $I_i$ is an interval. Assume without loss of generality that $I_1$ comes before $I_2$ and $I_2$ comes before $I_3$ on $Z$.

Set $T_1=S_1^r(I_1)$ and $T_2=S_2^r(I_2)$, and let $F_1,F_2\in \mathcal{C}_1$. By Lemma \ref{pairs}, $F_1\cap F_2$ intersects two of the $C_i$. Then, by Lemma \ref{sep}, $F_1\cap F_2$ intersects $T_1$ or $T_2$. It follows then our collection of $2$-intervals, $\mathcal{I}$, has matching number $1$.
\begin{figure}[h!]
    \centering
\begin{tikzpicture}[scale=.8]
\draw (0,0) circle (3cm);
\filldraw [black] (3,0) circle (2pt) node[right] {$f_0$};
\filldraw [black] (95:3cm) circle (2pt) node[above] {$\ \ \ \ \ f_3$};
\filldraw [black] (145:3cm) circle (2pt) node[left] {$f_2$};
\filldraw [black] (270:3cm) circle (2pt) node[below] {$f_1$};

\filldraw [black] (210:4cm) node[below] {$T_1=S_1^r(I_1)$};
\draw (210:4cm) -- (13:4cm);

\filldraw [black] (315:4cm) node[below] {$T_2=S_2^r(I_2)$};
\draw (315:4cm) -- (115:4cm);

\draw (3,0) -- (145:3cm);
\draw (95:3cm) -- (270:3cm);

\draw (315:1cm) circle [x radius=2cm, y radius=3mm, rotate=200];
\draw (325:.5cm) circle [x radius=5mm, y radius=2.4cm, rotate=30];
\draw (335:1.25cm) circle [x radius=5mm, y radius=2cm, rotate=-15];
\end{tikzpicture}
\caption{Case 1}
    \label{fig:case1}
\end{figure}

\bigskip\noindent
\textbf{Case 2.} One of the $C_i'$'s has two components (see Figure \ref{fig:Case2}).

We can assume without loss of generality that $C_3'$ has two components, and that $I_1$ comes before $I_2$ on $Z$.

\medskip 
{\bf Subcase 2.1.} If the order of the intervals on $Z$ is $I_1,I_2,I_3^1,I_3^2$, then set $T_1=S_1^r(I_1)$ and $T_2=S_2^r(I_2)$. 

\medskip 
{\bf Subcase 2.2.} If the order of the intervals is $I_1, I_3^1, I_2, I_3^2$, then set $T_1=S_3'$ and $T_2=S_1^r(I_1)$. 

\medskip 
{\bf Subcase 2.3.} If the order of the intervals is $I_1, I_3^1, I_3^2, I_2$, then set $T_1=S_1^r(I_1)$ and $T_2=S_2^l(I_2)$.

\medskip 
{\bf Subcase 2.4.} If the order of the intervals is $I_3^1, I_1,I_2,I_3^2$, then set $T_1=S_3'$ and $T_2=S_1^r(I_1)$.

The remaining subcases of Case 2 are symmetrical. For instance, the case where the order of the intervals is $I_3^1,I_3^2,I_1,I_2$ follows similarly to the case where the order of the intervals is $I_1,I_2,I_3^1,I_3^2$.

\begin{figure}[h!]
\centering
\begin{tikzpicture}[scale=.45]
\draw (0,0) circle (3cm);
\filldraw [black] (3,0) circle (2pt) node[right] {};
\filldraw [black] (95:3cm) circle (2pt) node[above] {};
\filldraw [black] (145:3cm) circle (2pt) node[left] {};
\filldraw [black] (270:3cm) circle (2pt) node[below] {};

%\filldraw [black] (210:4cm) node[below] {$T_1=S_1(r(I_1))$};
%\draw (210:4cm) -- (13:4cm);

%\filldraw [black] (315:4cm) node[below] {$T_2=S_2(r(I_2))$};
%\draw (315:4cm) -- (115:4cm);

\draw (3,0) -- (145:3cm);
\draw (95:3cm) -- (270:3cm);

\draw [blue] (280:1.3cm) circle [x radius=1.7cm, y radius=2.7mm, rotate=250];
\draw [red] (260:.8cm) circle [x radius=2mm, y radius=1.5cm, rotate=-40];
\draw [green] (335:0cm) circle [x radius=2mm, y radius=2cm, rotate=-60];
\end{tikzpicture}
\begin{tikzpicture}[scale=.45]
\draw (0,0) circle (3cm);
\filldraw [black] (3,0) circle (2pt) node[right] {};
\filldraw [black] (95:3cm) circle (2pt) node[above] {};
\filldraw [black] (145:3cm) circle (2pt) node[left] {};
\filldraw [black] (270:3cm) circle (2pt) node[below] {};

%\filldraw [black] (210:4cm) node[below] {$T_1=S_1(r(I_1))$};
%\draw (210:4cm) -- (13:4cm);

%\filldraw [black] (315:4cm) node[below] {$T_2=S_2(r(I_2))$};
%\draw (315:4cm) -- (115:4cm);

\draw (3,0) -- (145:3cm);
\draw (95:3cm) -- (270:3cm);

\draw [blue] (280:.8cm) circle [x radius=1.5cm, y radius=2.7mm, rotate=240];
\draw [red] (60:.4cm) circle [x radius=2mm, y radius=1.5cm, rotate=40];
\draw [green] (335:0cm) circle [x radius=2mm, y radius=2cm, rotate=-60];
\end{tikzpicture}
\begin{tikzpicture}[scale=.45]
\draw (0,0) circle (3cm);
\filldraw [black] (3,0) circle (2pt) node[right] {};
\filldraw [black] (95:3cm) circle (2pt) node[above] {};
\filldraw [black] (145:3cm) circle (2pt) node[left] {};
\filldraw [black] (270:3cm) circle (2pt) node[below] {};

%\filldraw [black] (210:4cm) node[below] {$T_1=S_1(r(I_1))$};
%\draw (210:4cm) -- (13:4cm);

%\filldraw [black] (315:4cm) node[below] {$T_2=S_2(r(I_2))$};
%\draw (315:4cm) -- (115:4cm);

\draw (3,0) -- (145:3cm);
\draw (95:3cm) -- (270:3cm);

\draw [blue] (280:.8cm) circle [x radius=1.5cm, y radius=2.7mm, rotate=240];
\draw [red] (190:-1.7cm) circle [x radius=2mm, y radius=1.15cm, rotate=-90];
\draw [green] (335:0cm) circle [x radius=2mm, y radius=2cm, rotate=-60];
\end{tikzpicture}
\begin{tikzpicture}[scale=.45]
\draw (0,0) circle (3cm);
\filldraw [black] (3,0) circle (2pt) node[right] {};
\filldraw [black] (95:3cm) circle (2pt) node[above] {};
\filldraw [black] (145:3cm) circle (2pt) node[left] {};
\filldraw [black] (270:3cm) circle (2pt) node[below] {};

%\filldraw [black] (210:4cm) node[below] {$T_1=S_1(r(I_1))$};
%\draw (210:4cm) -- (13:4cm);

%\filldraw [black] (315:4cm) node[below] {$T_2=S_2(r(I_2))$};
%\draw (315:4cm) -- (115:4cm);

\draw (3,0) -- (145:3cm);
\draw (95:3cm) -- (270:3cm);

\draw [blue] (0:.4cm) circle [x radius=2cm, y radius=2.7mm, rotate=140];
\draw [red] (175:-1.1cm) circle [x radius=2mm, y radius=1.5cm, rotate=15];
\draw [green] (335:1cm) circle [x radius=2mm, y radius=2cm, rotate=-60];
\end{tikzpicture}
 \caption{The configuration of each $C_i$ in the subcases of Case 2. Blue represents $C_1$, red represents $C_2$, and green represents $C_3$.}
    \label{fig:Case2}
\end{figure}
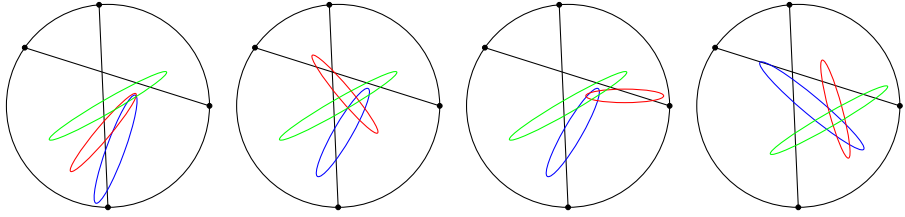

\bigskip\noindent
\textbf{Case 3.} Two of the $C_i'$'s have two components (see Figure \ref{fig:Case3}).

Without loss of generality, assume $C_2'$ and $C_3'$  have two components.

\medskip 
{\bf Subcase 3.1.} Assume the order of the intervals is $I_1,I_2^1,I_3^1, I_3^2, I_2^2$, then set $T_1=S_1^r(I_1)$ and $T_2=S_2'$.

\medskip 
{\bf Subcase 3.2.} If the order of the intervals is $I_1,I_2^1,I_3^1, I_2^2, I_3^2$, then set $T_1=S_1^r(I_1)$ and $T_2$ to be the line obtained by applying Lemma \ref{sssep} to $C_2$ and $C_3$.

%If there is a pair $H_1,H_2$ intersecting on or below $S_2'$ in $\overline{R^4}$, then $H_1\cap H_2$ must intersect $C_2$ and $C_3$ since $H_1\cap H_2$ intersects two of the $C_i$'s. There cannot be a pair $H_3,H_4$ lying  below $S_3'$ in $\overline{R^2}$ such that $H_3\cap H_4$ intersects $C_2'\cap \overline{R^2}$. If there were then three sets of $H_1,H_2,H_3$ and $H_4$ have a common point. If $H_1,H_2,H_3$ have a common point, then $H_3$ intersects $C_3'\cap \overline{R^4}$ and contains a point below $S_3'$ in $\overline{R^2}$, which implies that $H_3$ has a point in $R^1$. Similarly, $H_1,H_3,H_4$ cannot have a common point. Therefore, we can set $T_1=S_3'$ and $T_2=S_1(r(I_1))$.

\medskip 
{\bf Subcase 3.3.} If the order of the intervals is $I_2^1,I_1,I_3^1, I_3^2, I_2^2$, then set $T_1=S_1^r(I_1)$ and $T_2=S_2'$.

\medskip 
{\bf Subcase 3.4.} If the order of the intervals is $I_2^1,I_1,I_3^1, I_2^2, I_3^2$, then set $T_1=S^r(I_1)$ and $T_2$ to be the line obtained by applying Lemma \ref{sssep} to $C_2$ and $C_3$. Let $F$ be the intersection of a pair of sets in $\mathcal{C}_1$. If $F$ intersects $C_2$ and $C_3$, then $F$ intersects $T_2$ by Lemma \ref{sssep}. If $F$ intersects $C_1$ and $C_3$ or $C_1$ and $C_2'\cap \overline{R^4}$, then $F$ intersects $T_1$. If $T_2=S_2^r(I_2^1)$ and $F$ intersects $C_1$ and $C_2'\cap \overline{R^2}$, then $F$ intersects $T_2$ by Lemma \ref{sep}. If $T_2=S_3^l(I_3^2)$ and $F$ intersects $C_1$ and $C_2'\cap \overline{R^2}$, then $F$ intersects $T_2$ by Lemma \ref{sssep}.

\medskip 
{\bf Subcase 3.5.} If the order of the intervals is $I_2^1,I_3^1,I_1, I_3^2, I_2^2$, then set $T_1=S_2'$ and $T_2=S_3'$.

\medskip 
{\bf Subcase 3.6.} If the order of the intervals is $I_2^1,I_3^1,I_1, I_2^2, I_3^2$, then set $T_1=S_2'$ and $T_2=S_3'$.

The remaining subcases are symmetrical.

\begin{figure}[h!]
\centering
\begin{tikzpicture}[scale=.45]
\draw (0,0) circle (3cm);
\filldraw [black] (3,0) circle (2pt) node[right] {};
\filldraw [black] (95:3cm) circle (2pt) node[above] {};
\filldraw [black] (145:3cm) circle (2pt) node[left] {};
\filldraw [black] (270:3cm) circle (2pt) node[below] {};

%\filldraw [black] (210:4cm) node[below] {$T_1=S_1(r(I_1))$};
%\draw (210:4cm) -- (13:4cm);

%\filldraw [black] (315:4cm) node[below] {$T_2=S_2(r(I_2))$};
%\draw (315:4cm) -- (115:4cm);

\draw (3,0) -- (145:3cm);
\draw (95:3cm) -- (270:3cm);

\draw [blue] (100:-1.1cm) circle [x radius=2mm, y radius=1.5cm, rotate=160];
\draw [green] (335:0cm) circle [x radius=2mm, y radius=2cm, rotate=-60];
\draw [red] (245:1.5cm) -- (35:.5cm);
\draw [red] (15:2.4cm) -- (35:.5cm);
\draw [red] (245:1.5cm) -- (15:2.4cm);
\end{tikzpicture}
\begin{tikzpicture}[scale=.45]
\draw (0,0) circle (3cm);
\filldraw [black] (3,0) circle (2pt) node[right] {};
\filldraw [black] (95:3cm) circle (2pt) node[above] {};
\filldraw [black] (145:3cm) circle (2pt) node[left] {};
\filldraw [black] (270:3cm) circle (2pt) node[below] {};

%\filldraw [black] (210:4cm) node[below] {$T_1=S_1(r(I_1))$};
%\draw (210:4cm) -- (13:4cm);

%\filldraw [black] (315:4cm) node[below] {$T_2=S_2(r(I_2))$};
%\draw (315:4cm) -- (115:4cm);

\draw (3,0) -- (145:3cm);
\draw (95:3cm) -- (270:3cm);

\draw [blue] (100:-1.1cm) circle [x radius=2mm, y radius=1.5cm, rotate=160];
\draw [green] (335:0cm) circle [x radius=2mm, y radius=2cm, rotate=-70];
\draw [red] (245:1.5cm) -- (40:2.2cm);
\draw [red] (40:2.2cm) -- (35:.5cm);
\draw [red] (245:1.5cm) -- (35:.5cm);
\end{tikzpicture}
\begin{tikzpicture}[scale=.45]
\draw (0,0) circle (3cm);
\filldraw [black] (3,0) circle (2pt) node[right] {};
\filldraw [black] (95:3cm) circle (2pt) node[above] {};
\filldraw [black] (145:3cm) circle (2pt) node[left] {};
\filldraw [black] (270:3cm) circle (2pt) node[below] {};

%\filldraw [black] (210:4cm) node[below] {$T_1=S_1(r(I_1))$};
%\draw (210:4cm) -- (13:4cm);

%\filldraw [black] (315:4cm) node[below] {$T_2=S_2(r(I_2))$};
%\draw (315:4cm) -- (115:4cm);

\draw (3,0) -- (145:3cm);
\draw (95:3cm) -- (270:3cm);

\draw [blue] (80:-.85cm) circle [x radius=2mm, y radius=1.5cm, rotate=145];
\draw [green] (335:0cm) circle [x radius=2mm, y radius=2cm, rotate=-60];
\draw [red] (265:2.5cm) -- (35:.5cm);
\draw [red] (15:2.4cm) -- (35:.5cm);
\draw [red] (265:2.5cm) -- (15:2.4cm);
\end{tikzpicture}
\begin{tikzpicture}[scale=.45]
\draw (0,0) circle (3cm);
\filldraw [black] (3,0) circle (2pt) node[right] {};
\filldraw [black] (95:3cm) circle (2pt) node[above] {};
\filldraw [black] (145:3cm) circle (2pt) node[left] {};
\filldraw [black] (270:3cm) circle (2pt) node[below] {};

%\filldraw [black] (210:4cm) node[below] {$T_1=S_1(r(I_1))$};
%\draw (210:4cm) -- (13:4cm);

%\filldraw [black] (315:4cm) node[below] {$T_2=S_2(r(I_2))$};
%\draw (315:4cm) -- (115:4cm);

\draw (3,0) -- (145:3cm);
\draw (95:3cm) -- (270:3cm);

\draw [blue] (100:-.55cm) circle [x radius=2mm, y radius=1.5cm, rotate=130];
\draw [green] (335:0cm) circle [x radius=2mm, y radius=2cm, rotate=-70];
\draw [red] (260:2.5cm) -- (55:2.2cm);
\draw [red] (55:2.2cm) -- (35:.5cm);
\draw [red] (260:2.5cm) -- (35:.5cm);
\end{tikzpicture}

\begin{tikzpicture}[scale=.45]
\draw (0,0) circle (3cm);
\filldraw [black] (3,0) circle (2pt) node[right] {};
\filldraw [black] (95:3cm) circle (2pt) node[above] {};
\filldraw [black] (145:3cm) circle (2pt) node[left] {};
\filldraw [black] (270:3cm) circle (2pt) node[below] {};

%\filldraw [black] (210:4cm) node[below] {$T_1=S_1(r(I_1))$};
%\draw (210:4cm) -- (13:4cm);

%\filldraw [black] (315:4cm) node[below] {$T_2=S_2(r(I_2))$};
%\draw (315:4cm) -- (115:4cm);

\draw (3,0) -- (145:3cm);
\draw (95:3cm) -- (270:3cm);

\draw [blue] (50:.4cm) circle [x radius=2mm, y radius=1.5cm, rotate=200];
\draw [green] (335:0cm) circle [x radius=2mm, y radius=2cm, rotate=-60];
\draw [red] (265:2.5cm) -- (35:.5cm);
\draw [red] (15:2.4cm) -- (35:.5cm);
\draw [red] (265:2.5cm) -- (15:2.4cm);
\end{tikzpicture}
\begin{tikzpicture}[scale=.45]
\draw (0,0) circle (3cm);
\filldraw [black] (3,0) circle (2pt) node[right] {};
\filldraw [black] (95:3cm) circle (2pt) node[above] {};
\filldraw [black] (145:3cm) circle (2pt) node[left] {};
\filldraw [black] (270:3cm) circle (2pt) node[below] {};

%\filldraw [black] (210:4cm) node[below] {$T_1=S_1(r(I_1))$};
%\draw (210:4cm) -- (13:4cm);

%\filldraw [black] (315:4cm) node[below] {$T_2=S_2(r(I_2))$};
%\draw (315:4cm) -- (115:4cm);

\draw (3,0) -- (145:3cm);
\draw (95:3cm) -- (270:3cm);

\draw [blue] (50:.4cm) circle [x radius=2mm, y radius=1.5cm, rotate=200];
\draw [green] (335:0cm) circle [x radius=2mm, y radius=2cm, rotate=-70];
\draw [red] (245:1.5cm) -- (40:2.2cm);
\draw [red] (40:2.2cm) -- (35:.5cm);
\draw [red] (245:1.5cm) -- (35:.5cm);
\end{tikzpicture}
 \caption{The configuration of each $C_i$ in the subcases of Case 3.}
    \label{fig:Case3}

\end{figure}
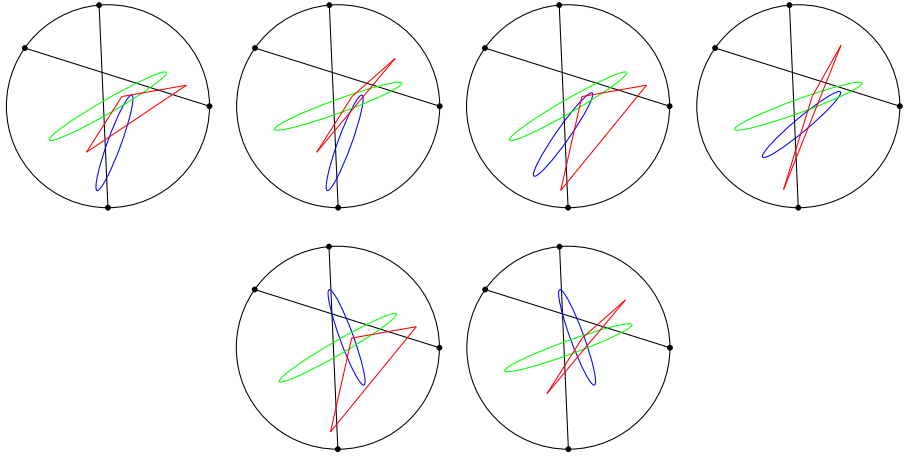

\bigskip\noindent
\textbf{Case 4:} Each $C_i'$ has two components (see Figure \ref{fig:Case4}).

\medskip 
{\bf Subcase 4.1.} If the order of the intervals is $I_1^1,I_2^1,I_3^1,I_3^2,I_2^2,I_1^2$, then set $T_1=S_1'$ and $T_2=S_2'$.

\medskip 
{\bf Subcase 4.2.} If the order of the intervals is $I_1^1,I_2^1,I_3^1,I_3^2,I_1^2,I_2^2$, then set $T_1=S_1'$ and $T_2=S_2'$.

\medskip 
{\bf Subcase 4.3.} If the order of the intervals is $I_1^1,I_2^1,I_3^1,I_1^2,I_2^2,I_3^2$, then set $T_1=S_2'$ and $T_2$ to be the line obtained by applying Lemma \ref{sssep} to $C_1$ and $C_3$. Let $F$ be the intersection of a pair of sets in $\mathcal{C}_1$. If $F$ intersects $C_1$ and $C_3$, then $H$ intersects $T_2$ by Lemma \ref{sssep}. If $F$ intersects $C_2'\cap \overline{R^4}$ and $C_1$ or $C_2'\cap \overline{R^4}$ and $C_3'\cap \overline{R^2}$, then $F$ intersects $T_1$ by Lemma \ref{ssep}. If $F$ intersects $C_2'\cap \overline{R^2}$ and $C_3$ or $C_2'\cap \overline{R^2}$ and $C_1'\cap \overline{R^4}$, then $H$ intersects $T_1$ by Lemma \ref{ssep}.

If $T_2=S_1^r(I_1^1)$ and $F$ intersects $C_3'\cap \overline{R^4}$, then $F$ intersects $T_2$ by Lemma \ref{sssep}. If $F$ intersects $C_2'\cap \overline{R^2}$ and $C_1'\cap \overline{R^2}$, then $F$ intersects $T_2$ by Lemma \ref{sep}.

If $T_2=S_3^l(I_3^1)$ and $F$ intersects $C_1'\cap \overline{R^2}$, then $F$ intersects $T_2$ by Lemma \ref{sssep}. If $F$ intersects $C_2'\cap \overline{R^4}$ and $C_3'\cap \overline{R^4}$, then $F$ intersects $T_2$ by Lemma \ref{sep}.

Therefore, the resulting family of $2$-intervals coming from these two lines has matching number $1$. 

%If $C_1'\cap \overline{R^2}$ or $C_3'\cap \overline{R^4}$ consists of a single point, then set $S_1(r(I_1^1))=L_1$ or $S_3(l(I_3^2))=L_2$, respectively. This ensures that $S_1(r(I_1^1))$ always intersects $[c,f_0]$ and $S_3(l(I_3^2))$ intersects $[f_1,c]$. 

\medskip 
{\bf Subcase 4.4.} If the order of the intervals is $I_1^1,I_2^1,I_3^1,I_2^2,I_3^2,I_1^2$, then set $T_1=S_1'$ and $T_2$ to be the line obtained by applying Lemma \ref{sssep} to $C_2$ and $C_3$. 

%Let $X=S_2'$. Similar to the above argument, there cannot be a pair of sets intersecting below $X$ in $\overline{R^2}$ and a pair of sets intersecting below $X$ in $\overline{R^4}$. If there is no pair of sets intersecting below $X$ in $\overline{R^4}$, then set $T_1=X$ and $T_2=S_1(r(I_1^1))$. Otherwise, let $Y=S_3'$. Again, similarly, there cannot be a pair of sets lying below $Y$ in $\overline{R^2}$ and a pair of sets lying below $Y$ in $\overline{R^4}$. If there is no pair of sets lying below $Y$ in $\overline{R^2}$, then set $T_1=Y$ and $T_2=S_1(l(I_1^2))$. If there is no pair of sets lying below $Y$ in $\overline{R^4}$, then set $T_1=X$ and $T_2=Y$.

\medskip 
{\bf Subcase 4.5.} If the order of the intervals is $I_1^1,I_2^1,I_3^1,I_2^2,I_1^2,I_3^2$, then set $T_1=S_1'$ and $T_2$ to be the line obtained by applying Lemma \ref{sssep} to $C_2$ and $C_3$. A similar argument as in subcase 3 shows that the resulting family of $2$-intervals coming from these two lines has matching number $1$.

\begin{figure}[h!]
    \centering

\begin{tikzpicture}[scale=.45]
\draw (0,0) circle (3cm);
\filldraw [black] (3,0) circle (2pt) node[right] {};
\filldraw [black] (95:3cm) circle (2pt) node[above] {};
\filldraw [black] (145:3cm) circle (2pt) node[left] {};
\filldraw [black] (270:3cm) circle (2pt) node[below] {};

%\filldraw [black] (210:4cm) node[below] {$T_1=S_1(r(I_1))$};
%\draw (210:4cm) -- (13:4cm);

%\filldraw [black] (315:4cm) node[below] {$T_2=S_2(r(I_2))$};
%\draw (315:4cm) -- (115:4cm);

\draw (3,0) -- (145:3cm);
\draw (95:3cm) -- (270:3cm);

\draw [blue] (265:2.5cm) -- (35:.5cm);
\draw [blue] (5:2.8cm) -- (35:.5cm);
\draw [blue] (265:2.5cm) -- (5:2.8cm);
\draw [green] (335:0cm) circle [x radius=2mm, y radius=2cm, rotate=-60];
\draw [red] (245:1.5cm) -- (35:.5cm);
\draw [red] (15:2.4cm) -- (35:.5cm);
\draw [red] (245:1.5cm) -- (15:2.4cm);
\end{tikzpicture}
\begin{tikzpicture}[scale=.45]
\draw (0,0) circle (3cm);
\filldraw [black] (3,0) circle (2pt) node[right] {};
\filldraw [black] (95:3cm) circle (2pt) node[above] {};
\filldraw [black] (145:3cm) circle (2pt) node[left] {};
\filldraw [black] (270:3cm) circle (2pt) node[below] {};

%\filldraw [black] (210:4cm) node[below] {$T_1=S_1(r(I_1))$};
%\draw (210:4cm) -- (13:4cm);

%\filldraw [black] (315:4cm) node[below] {$T_2=S_2(r(I_2))$};
%\draw (315:4cm) -- (115:4cm);

\draw (3,0) -- (145:3cm);
\draw (95:3cm) -- (270:3cm);

\draw [blue] (265:2.5cm) -- (35:.5cm);
\draw [blue] (20:2.6cm) -- (35:.5cm);
\draw [blue] (265:2.5cm) -- (20:2.6cm);
\draw [green] (335:0cm) circle [x radius=2.3mm, y radius=2cm, rotate=-50];
\draw [red] (245:1.5cm) -- (35:.5cm);
\draw [red] (5:2.8cm) -- (35:.5cm);
\draw [red] (245:1.5cm) -- (5:2.8cm);
\end{tikzpicture}
\begin{tikzpicture}[scale=.45]
\draw (0,0) circle (3cm);
\filldraw [black] (3,0) circle (2pt) node[right] {};
\filldraw [black] (95:3cm) circle (2pt) node[above] {};
\filldraw [black] (145:3cm) circle (2pt) node[left] {};
\filldraw [black] (270:3cm) circle (2pt) node[below] {};

%\filldraw [black] (210:4cm) node[below] {$T_1=S_1(r(I_1))$};
%\draw (210:4cm) -- (13:4cm);

%\filldraw [black] (315:4cm) node[below] {$T_2=S_2(r(I_2))$};
%\draw (315:4cm) -- (115:4cm);

\draw (3,0) -- (145:3cm);
\draw (95:3cm) -- (270:3cm);

\draw [blue] (265:2.5cm) -- (35:.3cm);
\draw [blue] (63:2.8cm) -- (35:.3cm);
\draw [blue] (265:2.5cm) -- (63:2.8cm);
\draw [green] (345:.5cm) circle [x radius=2mm, y radius=2cm, rotate=-70];
\draw [red] (255:1.7cm) -- (35:2.2cm);
\draw [red] (35:2.2cm) -- (35:.7cm);
\draw [red] (255:1.7cm) -- (35:.7cm);
\end{tikzpicture}
\begin{tikzpicture}[scale=.45]
\draw (0,0) circle (3cm);
\filldraw [black] (3,0) circle (2pt) node[right] {};
\filldraw [black] (95:3cm) circle (2pt) node[above] {};
\filldraw [black] (145:3cm) circle (2pt) node[left] {};
\filldraw [black] (270:3cm) circle (2pt) node[below] {};

%\filldraw [black] (210:4cm) node[below] {$T_1=S_1(r(I_1))$};
%\draw (210:4cm) -- (13:4cm);

%\filldraw [black] (315:4cm) node[below] {$T_2=S_2(r(I_2))$};
%\draw (315:4cm) -- (115:4cm);

\draw (3,0) -- (145:3cm);
\draw (95:3cm) -- (270:3cm);

\draw [blue] (265:2.5cm) -- (35:.5cm);
\draw [blue] (5:2.8cm) -- (35:.5cm);
\draw [blue] (265:2.5cm) -- (5:2.8cm);
\draw [green] (335:0cm) circle [x radius=2mm, y radius=2cm, rotate=-70];
\draw [red] (245:2cm) -- (50:2.2cm);
\draw [red] (50:2.2cm) -- (35:0.2cm);
\draw [red] (245:2cm) -- (35:0.2cm);
\end{tikzpicture}

\begin{tikzpicture}[scale=.45]
\draw (0,0) circle (3cm);
\filldraw [black] (3,0) circle (2pt) node[right] {};
\filldraw [black] (95:3cm) circle (2pt) node[above] {};
\filldraw [black] (145:3cm) circle (2pt) node[left] {};
\filldraw [black] (270:3cm) circle (2pt) node[below] {};

%\filldraw [black] (210:4cm) node[below] {$T_1=S_1(r(I_1))$};
%\draw (210:4cm) -- (13:4cm);

%\filldraw [black] (315:4cm) node[below] {$T_2=S_2(r(I_2))$};
%\draw (315:4cm) -- (115:4cm);

\draw (3,0) -- (145:3cm);
\draw (95:3cm) -- (270:3cm);

\draw [blue] (265:2.5cm) -- (0:0.2cm);
\draw [blue] (40:2.8cm) -- (0:0.2cm);
\draw [blue] (265:2.5cm) -- (40:2.8cm);
\draw [green] (0:.5cm) circle [x radius=2mm, y radius=2cm, rotate=-80];
\draw [red] (245:2cm) -- (50:2.2cm);
\draw [red] (50:2.2cm) -- (35:0.2cm);
\draw [red] (245:2cm) -- (35:0.2cm);
\end{tikzpicture}
 \caption{The configuration of each $C_i$ in the subcases of Case 4}
    \label{fig:Case4}

\end{figure}
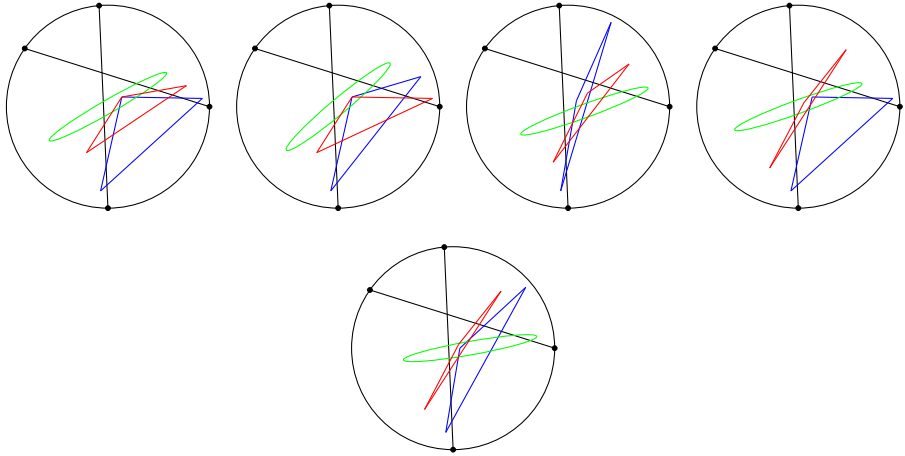

\newpage

\section{Acknowledgement}
The author would like to thank Shira Zerbib for many helpful remarks and for improving the overall presentation of this paper.

%\bibliographystyle{plain}
%\bibliography{mybibliography}

\end{document}